\documentclass[11pt]{article}%
\usepackage{mathrsfs}
\usepackage{bbm}
\usepackage{amsfonts}
\usepackage{amsmath,amssymb}
\usepackage{amsmath}
\usepackage{amssymb}
\usepackage{graphicx}%
\usepackage{shorttoc}
\providecommand{\U}[1]{\protect \rule{.1in}{.1in}}

\newtheorem{theorem}{Theorem}[section]

\newtheorem{definition}[theorem]{Definition}

\newtheorem{lemma}[theorem]{Lemma}
\newtheorem{proposition}[theorem]{Proposition}
\newtheorem{remark}[theorem]{Remark}

\newenvironment{proof}[1][Proof]{\noindent \textbf{#1.} }{\  $\Box$}
\numberwithin{equation}{section}

\begin{document}

\title{\textbf{Stein's Method for Law of Large Numbers under Sublinear  Expectations  }}
\author{ Yongsheng Song\thanks{RCSDS, Academy of Mathematics and Systems Science, Chinese Academy of Sciences, Beijing 100190, China, and
School of Mathematical Sciences, University of Chinese Academy of Sciences, Beijing 100049, China. E-mail:
yssong@amss.ac.cn.}  }

\date{\today}
\maketitle

\begin{abstract}  Peng, S. (\cite{P08b}) proved the law of large numbers under a sublinear expectation. In this paper, we give its error estimates  by Stein's method.
\end{abstract}


\textbf{MSC-classification}: 60F05; 60G50

\section{Introduction}
Peng, S. (\cite{P08b}) proved the law of large numbers (LLN) under a sublinear expectation:

 Let $(X_k)_{k\ge 1}$ be a sequence of independent and identically distributed random variables  under a sublinear expectation $\hat{\mathbf{E}}$ with $\hat{\mathbf{E}}[|X_1|^{2}]<\infty$. Set $S_n=\frac{X_1+\cdots+X_n}{n}$, $\underline{\mu}=-\hat{\mathbf{E}}[-X_1]$, $\overline{\mu}=\hat{\mathbf{E}}[X_1]$. 
 \begin {eqnarray}
\lim_{n} \hat{\mathbf{E}}[\phi(S_n)]=\sup_{y\in[\underline{\mu}, \overline{\mu}]}\phi(y),
\end {eqnarray} for any $\phi\in C_{b, Lip}(\mathbb{R})$, the collection of bounded and Lipschitz continuous functions on $\mathbb{R}$. The sublinear expectation $\mathcal{M}_p[\phi]:=\sup_{y\in[\underline{\mu}, \overline{\mu}]}\phi(y)$ is called a maximal distribution.

In this paper, we give the following error estimate for Peng's law of large numbers by Stein's method:
\begin {eqnarray}
\sup_{|\phi|_{Lip}\le1}\bigg| \hat{\mathbf{E}}[\phi(S_n)]-\sup_{y\in[\underline{\mu}, \overline{\mu}]}\phi(y)\bigg|\le C n^{-1/2},
\end {eqnarray} where $C$ is a constant depending only on $\hat{\mathbf{E}}[|X_1|^2]$.

For the linear case,  Stein's method, which made its first appearance in the ground
breaking work of Stein (1972), is a powerful tool to estimate  the error of normal approximation. In Song (2017) (\cite {Song17}), Stein's method under sublinear expectations was established to give error estimates for Peng's central limit theorem under a sublinear expectation.

In this paper,  we shall establish Stein's method for LLN under sublinear expectations. 

Let $\mathcal{N}[\phi]=\sup_{\mu\in\Theta}\mu[\phi]$ be a  sublinear expectation on $C_{b,Lip}(\mathbb{R})$  with $\underline{\mu}=-\mathcal{N}[-x]$, $\overline{\mu}=\mathcal{N}[x]$ and $\mathcal{N}[|x|^{1+\alpha}]<\infty$ for some $\alpha \in (0,1] $. For $\phi\in C_{b,Lip}(\mathbb{R})$, define $v(x,t):=\sup_{y\in[\underline{\mu}, \overline{\mu}]}\phi(x+ty).$  Then $v$ is the viscosity solution to the equation below:

\begin {eqnarray} \label {pE-intro}
\begin {split}
\partial_tv(x,t)-p(\partial_xv(x,t)) & = 0, \ (x,t)\in \mathbb{R}\times (0,\infty),\\
v(x,0) & =\phi(x),
\end {split}
\end {eqnarray}
where $p(a)=\sup_{y\in[\underline{\mu},\overline{\mu}]}(y a)$, $a\in\mathbb{R}$. Set $\phi_s(x):=v((1-s)x,s)$ and $w(s):=\mathcal{N}[\phi_s]$. Then $w(1)=v(0,1)=\sup_{y\in[\underline{\mu}, \overline{\mu}]}\phi(y)$, and $w(0)=\mathcal{N}[\phi]$.

SUPPOSE that $v$ belongs to $C_b^{1,\alpha}(\mathbb{R}\times\mathbb{R}_+)$. It can be shown that, for a.e. $s\in(0,1)$,
\begin {eqnarray}\label {intro-derivative}
w^{\prime}(s)=\frac{1}{1-s}\mu_s[p(\phi'_s(x))-x\phi'_s(x)],
\end {eqnarray} where  $\mu_s\in\Theta$ with $\mu_s[\phi_s]=\mathcal{N}[\phi_s]$.  From this, we get a substitute of the Stein equation.

\textbf{Step 1.} $\sup_{y\in[\underline{\mu}, \overline{\mu}]}\phi(y)-\mathcal{N}[\phi]=\int_0^1\frac{1}{1-s}\mu_s[p(\phi'_s(x))-x\phi'_s(x)] ds$.

Now the next job is to calculate the expectation on the right side of the equality (\ref {intro-derivative}).
 
 For $\phi\in C_b^{1,\alpha}(\mathbb{R})$ and $\mu\in\Theta$ with $\mu[\phi]=\mathcal{N}[\phi]$, we have

\textbf{Step 2.} $\bigg|\mu[p(\phi'(x))-x\phi'(x)]\bigg|\le 4 [\phi']_{\alpha} \mathcal{N}[|x|^{1+\alpha}],$ where $p(a)=\mathcal{N}[ax]$, $a\in \mathbb{R}$.

Following Step 1 and Step 2, we give the rate of convergence of Peng's LLN.

 Let $(X_k)_{k\ge 1}$ be a sequence of i.i.d random variables  under a sublinear expectation $\hat{\mathbf{E}}$ with $\hat{\mathbf{E}}[|X_1|^{2}]<\infty$. Set $S_n=\frac{X_1+\cdots+X_n}{n}$, $\underline{\mu}=-\hat{\mathbf{E}}[-X_1]$, $\overline{\mu}=\hat{\mathbf{E}}[X_1]$. 
 
\textbf{Step 3.} $\bigg|\mathbb{E}[\phi(S_n)]-\sup_{\mu\in[\underline{\mu}, \overline{\mu}]}\phi(\mu)\bigg|\le   4n^{-\alpha} \int_{0}^{1}[\partial_xv(\cdot, s)]_\alpha ds \times \mathbb{E}[|X_1|^{1+\alpha}].$

Note that the arguments in Step 1 and Step 3 are based on the assumption that $v\in C_b^{1,\alpha}(\mathbb{R}\times\mathbb{R}_+)$ or $\int_{0}^{1}[\partial_xv(\cdot, s)]_\alpha ds<\infty.$ Unfortunately, generally the solution $v$ to Equ. (\ref {pE-intro}) is only Lipschitz continuous. In order to give error estimates for Peng's LLN, we need to consider proper smooth approximations of the solution $v$, which turn out to be solutions to the non-homogeneous equations below.
\begin {eqnarray} \label {gpE-intro}
\begin {split}
\partial_tv(x,t)-p(\partial_xv(x,t)) & = f(x,t), \ (x,t)\in \mathbb{R}\times (0,\infty),\\
v(x,0) & =\phi(x).
\end {split}
\end {eqnarray}
In Section 3, we give modified arguments of Step 1--Step 3 corresponding to Equ. (\ref {gpE-intro}), which establish the Stein's method for LLN under sublinear expectations. In Section 4,  we construct two sequences of smooth approximations of the solutions to Equ. (\ref {pE-intro}), which are motivated by Krylov (2018) (\cite {Kr}). In Section 5, we prove the error estimates for Peng's LLN. In Section 6, we give a Stein type characterization for maximal distribution $\mathcal {M}_p$.

\section{Basic Notions of Sublinear Expectations}
Here we review basic notions and results of sublinear expectations.

Let $\Omega$ be a given set and let $\mathcal{H}$ be a  linear space of real valued functions defined on $\Omega$ such that for any $X\in\mathcal{H}$ and $\varphi\in C_{b,Lip}(\mathbb{R})$, we have $\varphi(X)\in \mathcal{H}$.
The space $\mathcal{H}$ is considered as our space of random variables.

\begin{definition}
A  sublinear expectation  is a functional $\hat{\mathbf{E}}: \mathcal{H}\to \mathbb{R}$ satisfying
\begin{description}
\item[E1.]  $\hat{\mathbf{E}}[X]\geq \hat{\mathbf{E}}[Y],\ \text{if}\ X\ge Y$;

\item[E2.]  $\hat{\mathbf{E}}[\lambda X]=\lambda \hat{\mathbf{E}}[X],\ \text{for}\ \lambda\geq 0$;

\item[E3.]  $\hat{\mathbf{E}}[c]=c, \ \text{for}\ c\in \mathbb{R}$;

\item[E4.]  $\hat{\mathbf{E}}[X+Y]\leq \hat{\mathbf{E}}[X]+\hat{\mathbf{E}}[Y]$, \ for  $X, Y\in \mathcal{H}$;

\item[E5.]  $\hat{\mathbf{E}}[\varphi_n(X)]\downarrow0$, for $X\in\mathcal{H}$ and $\varphi_n\in C_{b,Lip}(\mathbb{R})$, $\varphi_n\downarrow0$.
\end{description}
\end{definition}
The triple $(\Omega, \mathcal{H}, \hat{\mathbf{E}})$ is called a  sublinear expectation space. For $X\in\mathcal{H}$, set \[\mathcal{N}^X[\varphi]=\hat{\mathbf{E}}[\varphi(X)],  \ \varphi\in C_{b,Lip}(\mathbb{R}),\]  which is a sublinear expectation on $C_{b,Lip}(\mathbb{R})$. We say $X$ is distributed as  $\mathcal{N}^X$, write $X\sim\mathcal{N}^X$.  A functional $\mathcal{N}$ is a sublinear expectation  on $C_{b,Lip}(\mathbb{R})$ if and only if it can be represented as the
supremum expectation of a weakly compact subset $\Theta$ of probability measures  on $(%
\mathbb{R},\mathcal{B}(\mathbb{R}))$ (see \cite{DHP11}),
\begin {eqnarray}
\mathcal{N}[\varphi]=\sup_{\mu\in\Theta}\mu[\varphi], \ \textmd{for all} \  \varphi\in C_{b,Lip}(\mathbb{R}).
\end {eqnarray}

\begin {definition} Let  $(\Omega, \mathcal{H}, \hat{\mathbf{E}})$ be a sublinear expectation space. We say a random vector $\mathbf{ X}=(X_1,\cdots, X_m)\in\mathcal{H}^m$ is independent from $\mathbf{ Y}=(Y_1,\cdots, Y_n)\in\mathcal{H}^n$ if for any $\varphi\in C_{b,Lip}(\mathbb{R}^{n+m})$
\[\hat{\mathbf{E }}[\varphi(\mathbf{ Y},\mathbf{ X})]=\hat{\mathbf{ E}}[\hat{\mathbf{ E}}[\varphi(\mathbf{ y}, \mathbf{ X})]|_{\mathbf{ y}=\mathbf{ Y}}].\]
\end {definition}
In a sublinear expectation space,  the fact that $\mathbf{ X}$ is independent from $\mathbf{ Y}$ does not imply that $\mathbf{Y}$ is independent from $\mathbf{X}$. We say $(X_i)_{i\ge1}$ is a sequence of independent random variables means that $X_{i+1}$ is independent from $(X_1,\cdots, X_i)$ for each $i\in\mathbb{N}$.

\begin {definition} Let $(\Omega, \mathcal{H}, \hat{\mathbf{E}})$ and   $(\widetilde{\Omega}, \mathcal{\widetilde{H}}, \widetilde{\mathbf{E}})$ be two sublinear expectations. A random vector $\mathbf{ X}$ in  $(\Omega, \mathcal{H}, \hat{\mathbf{E}})$ is said to be identically distributed with another random vector $\mathbf{ Y}$ in $(\widetilde{\Omega}, \mathcal{\widetilde{H}}, \widetilde{\mathbf{E}})$ (write $\mathbf{ X}\mathop{=}\limits^d\mathbf{ Y}$), if for any bounded and Lipschitz function $\varphi$, \[\hat{\mathbf{ E}}[\varphi(X)]=\widetilde{\mathbf{ E}}[\varphi(Y)].\]

\end {definition}

\section {Stein's method for LLN under sublinear expectations}
Let $\mathcal{N}[\varphi]=\sup_{\mu\in\Theta}\mu[\varphi]$ be a sublinear expectation on  $C_{b,Lip}(\mathbb{R})$. Throughout this article, we suppose the following additional property:
\begin{description}
\item[(H)] $\lim_{N\rightarrow\infty}\mathcal{N}[|x|1_{[|x|>N]}]=0.$
\end{description}
Note that the condition $(H)$ is naturally satisfied if $\mathcal{N}[|x|^{1+\delta}]<\infty$ for some $\delta>0$.

Define $\xi: \mathbb{R}\rightarrow\mathbb{R}$ by $\xi(x)=x$. Sometimes, we write $\mathcal{N}_G[\varphi], \ \mathcal{N}[\varphi]$ and  $\mu[\varphi]$ by $\mathbb{E}_G[\varphi(\xi)], \ \mathbb{E}[\varphi(\xi)]$ and $E_\mu[\varphi(\xi)]$, respectively. For  $\varphi\in C_{b,Lip}(\mathbb{R})$, set $\Theta_\varphi=\{\mu\in \Theta:E_{\mu}[\varphi(\xi)]=\mathbb{E}[\varphi(\xi)]\}$.

\begin{lemma}   \label {SteinEquation-LLN}For $\phi\in C_{b,Lip}(\mathbb{R})$, let $v\in C_b^{1,\alpha}(\mathbb{R}\times\mathbb{R}_+)$ with some $\alpha\in (0,1]$, be the solution to the following equation:
\begin {eqnarray} \label {gpE}
\begin {split}
\partial_tv(x,t)-p(\partial_xv(x,t)) & = f(x,t), \ (x,t)\in \mathbb{R}\times (0,\infty),\\
v(x,0) & =\phi(x),
\end {split}
\end {eqnarray}
where $p(a)=\sup_{y\in[\underline{\mu},\overline{\mu}]}(y a)$, $a\in\mathbb{R}$, for two real numbers $\underline{\mu}\le\overline{\mu}$.

Set $\phi_s(x):=v( (1-s)x,s)$. Then
\begin{eqnarray}\label {SteinEquation}v(0,1)-\mathcal{N}[\phi]=\int_0^1\bigg(\frac{1}{1-s}E_{\mu_s}[\mathcal{L}_p\phi_s(\xi)]+E_{\mu_s}[f((1-s)\xi,s)]\bigg)ds,
\end {eqnarray} where $\mathcal{L}_p\phi_s(x)=p(\phi_s'(x))-x\phi_s'(x)$, $\mu_s\in\Theta_{\phi_s}$. \end{lemma}

\begin{proof} Set $w(s)=\mathbb{E}[v( (1-s)\xi,s)]$. Then $w(1)=v(0,1)$ and $w(0)=\mathcal{N}[\phi]$.
By Lemma 2.2 in Hu, Peng and Song (2017), we have, for $s\in (0,1)$,
\begin{eqnarray*}
\partial_s^+w(s):&=&\lim_{\delta\rightarrow0+}\frac{w(s+\delta)-w(s)}{\delta}\\
&=&\sup_{\mu_s\in\Theta_{\phi_s}}E_{\mu_s}[\partial_s\phi_s(\xi)]\\
&=&\sup_{\mu_s\in\Theta_{\phi_s}}E_{\mu_s}[\frac{1}{1-s}\mathcal{L}_p\phi_s(\xi)+f((1-s)\xi,s)]]
\end{eqnarray*}
and
\begin{eqnarray*}
\partial_s^-w(s):&=&\lim_{\delta\rightarrow0+}\frac{w(s-\delta)-w(s)}{-\delta}\\
&=&\inf_{\mu_s\in\Theta_{\phi_s}}E_{\mu_s}[\partial_s\phi_s(\xi)]\\
&=&\inf_{\mu_s\in\Theta_{\phi_s}}E_{\mu_s}[\frac{1}{1-s}\mathcal{L}_p\phi_s(\xi)+f((1-s)\xi,s)]]
\end {eqnarray*}

Noting
that $w$ is continuous on $[0,1]$ and locally Lipschitz continuous on $(0,1)$ by the supposed regularity properties of $v$, we have $w'(s)=\partial_s^+w(s)=\partial_s^-w(s)$  for a.e. $s\in(0,1)$ and consequently \[w(1)-w(0)=\int_0^1\partial_s^+w(s)ds=\int_0^1\partial_s^-w(s)ds.\]
\end{proof}

The next Lemma gives an estimate of the expectations on the right hand of  Equ. (\ref{SteinEquation}).

\begin {lemma} \label {SteinEstimate} Let $\alpha\in(0,1]$. Suppose  $\mathbb{E}[|\xi|^{1+\alpha}]<\infty$. For $\phi\in C_b^{1,\alpha}(\mathbb{R})$ and $\mu\in\Theta_{\phi}$, we have
\begin {eqnarray}
\bigg|E_{\mu}[\xi\phi'(\xi)-p_\xi(\phi'(\xi))]\bigg|\le 4 [\phi']_{\alpha} \mathbb{E}[|\xi|^{1+\alpha}],
\end {eqnarray}
where $p_\xi(a)=\mathbb{E}[a\xi]$, $a\in \mathbb{R}$.
\end {lemma}
\begin {proof} Taylor's  formula gives
\begin{eqnarray}
\label {Taylor1} \phi(\xi)&=&\phi(0)+\phi'(0)\xi+R_{\xi},\\
\label {Taylor2} \phi'(\xi)&=&\phi'(0)+R'_{\xi},
\end{eqnarray}
 with   $|R_{\xi}|\le[\phi']_{\alpha} |\xi|^{1+\alpha}$ and $|R'_{\xi}|\le[\phi']_{\alpha} |\xi|^{\alpha}.$

  Set $A:=\mathbb{E}[\phi(\xi)]=E_{\mu}[\phi(\xi)]$. Then
\begin {eqnarray*}A=\mathbb{E}[\phi(\xi)]&=&\mathbb{E}[\phi(0)+\phi'(0)\xi+R_{\xi}]\\
&\le& \phi(0)+\mathbb{E}[\phi'(0)\xi]+\mathbb{E}[R_{\xi}]\\
&\le& \phi(0)+p_{\xi}(\phi'(0))+[\phi']_{\alpha}\mathbb{E}[|\xi|^{1+\alpha}],
\end {eqnarray*} and
\begin {eqnarray*}A=\mathbb{E}[\phi(\xi)]&=&\mathbb{E}[\phi(0)+\phi'(0)\xi+R_{\xi}]\\
&\ge& \phi(0)+\mathbb{E}[\phi'(0)\xi]-\mathbb{E}[-R_{\xi}]\\
&\ge& \phi(0)+p_{\xi}(\phi'(0))-[\phi']_{\alpha}\mathbb{E}[|\xi|^{1+\alpha}].
\end {eqnarray*}
Therefore, \[\bigg|A-\phi(0)-p_{\xi}(\phi'(0))\bigg|\le[\phi']_{\alpha}\mathbb{E}[|\xi|^{1+\alpha}].\]
Noting that $A=E_{\mu}[\phi(\xi)]=\phi(0)+\phi'(0)E_{\mu}[\xi]+E_{\mu}[R_{\xi}]$, we have
\begin {eqnarray}
\label {estimate1}
\bigg|\phi'(0)E_{\mu}[\xi]-p_{\xi}(\phi'(0))\bigg|= \bigg|A-\phi(0)-E_{\mu}[R_{\xi}]-p_{\xi}(\phi'(0))\bigg|\le 2[\phi']_{\alpha}\mathbb{E}[|\xi|^{1+\alpha}].
\end {eqnarray}
Now let us compute the expectation $E_{\mu}[\xi\phi'(\xi)-p_{\xi}(\phi'(\xi))]$. By (\ref {Taylor2}), we have
\begin {eqnarray*}
& &\xi\phi'(\xi)-p_{\xi}(\phi'(\xi))\\
&=&\xi(\phi'(0)+R'_{\xi})-p_{\xi}(\phi'(0)+R'_{\xi})\\
&=&[\xi\phi'(0)-p_{\xi}(\phi'(0))]+[p_{\xi}(\phi'(0))-p_{\xi}(\phi'(0)+R'_{\xi})]+\xi R'_{\xi}.
\end {eqnarray*} So, by (\ref{estimate1}),
\begin {eqnarray*}
& &\bigg|E_{\mu}[\xi\phi'(\xi)-p_{\xi}(\phi'(\xi))]\bigg|\\
&=&\bigg|E_{\mu}[\xi\phi'(0)-p_{\xi}(\phi'(0))]+E_{\mu}[p_{\xi}(\phi'(0))-p_{\xi}(\phi'(0)+R'_{\xi})]+E_{\mu}[\xi R'_{\xi}]\bigg|\\
&\le&2[\phi']_{\alpha}\mathbb{E}[|\xi|^{1+\alpha}]+[\phi']_{\alpha}\mathbb{E}[|\xi|]\mathbb{E}[|\xi|^{\alpha}]+[\phi']_{\alpha}\mathbb{E}[|\xi|^{1+\alpha}]\\
&\le&4[\phi']_{\alpha}\mathbb{E}[|\xi|^{2+\alpha}].
\end {eqnarray*}
\end {proof}

\begin {proposition} \label {Stein-pro} Let $v\in C_b^{1,\alpha}(\mathbb{R}\times\mathbb{R}_+)$ for some $\alpha\in (0,1]$, be a solution to Equ. (\ref{gpE}) with $p(a)=\mathbb{E}[a\xi]$. For $0\le t \le \bar{t}\le 1$, set  $\delta=\bar{t}-t$.
Then
\begin {eqnarray}
\begin {split} \label {Est1-local}
v(0,\bar{t})-\mathbb{E}[v( \delta\xi,t)]\ge & - 4\delta^{\alpha} \int_{t}^{\bar{t}}[\partial_xv(\cdot, s)]_\alpha ds \times \mathbb{E}[|\xi|^{1+\alpha}]\\
& -\int_t^{\bar{t}}\mathbb{E}[-f((\bar{t}- s)\xi, s)] ds,
\end {split}
\end {eqnarray}
and
\begin {eqnarray}
\begin {split}
v(0,\bar{t})-\mathbb{E}[v( \sqrt{\delta}\xi,t)]\le  & 4\delta^{\alpha} \int_{t}^{\bar{t}}[\partial_xv(\cdot, s)]_\alpha ds \times \mathbb{E}[|\xi|^{1+\alpha}]\\
&+ \int_t^{\bar{t}}\mathbb{E}[f((\bar{t}- s)\xi, s)] ds.
\end {split}
\end {eqnarray}
\end{proposition}

\begin {proof}
For $0\le t \le \bar{t}\le 1$, set $v_{t,\bar{t}}(x,s)=v(\delta x, \delta s+t)$, $(x,s)\in \mathbb{R}\times [0,1]$. Then $v_{t,\bar{t}}$ is the solution to the equation below:
\begin {eqnarray*}
\partial_sv_{t,\bar{t}}(x,s)-p(\partial_xv_{t,\bar{t}}(x,s)) &=& \delta f(\delta x,t+\delta s), \ (x,s)\in \mathbb{R}\times (0,1],\\
v_{t,\bar{t}}(x,0) &=&v(\delta x, t).
\end {eqnarray*}
By Lemma \ref{SteinEquation-LLN}, we get
\begin {eqnarray*}& &v(0,t+\delta)-\mathbb{E}[v( \delta\xi,t)]\\
&=&v_{t,\bar{t}}(0,1)-\mathbb{E}[v_{t,\bar{t}}(\xi,0)]\\
&=&\int_0^1\frac{1}{1-s}E_{\mu^{t,\bar{t}}_{\delta s+t}}[\mathcal{L}_p\phi^{t,\bar{t}}_{\delta s+t}(\xi)]+E_{\mu^{t,\bar{t}}_{\delta s+t}}[\delta f(\delta(1-s)\xi, t+\delta s)] ds\\
&=&\int_t^{\bar{t}}\frac{1}{\bar{t}-s}E_{\mu^{t,\bar{t}}_{s}}[\mathcal{L}_p\phi^{t,\bar{t}}_{s}(\xi)]+E_{\mu^{t,\bar{t}}_{s}}[ f((\bar{t}- s)\xi, s)] ds.
\end {eqnarray*}
Here $\phi^{t,\bar{t}}_s(x)=v((\bar{t}-s)x, s)$, $\mu^{t,\bar{t}}_s\in\Theta_{\phi^{t,\bar{t}}_s}$,  $s\in[t, \bar{t}]$.

 By Lemma \ref {SteinEstimate}, we have
\begin {eqnarray*}& &\bigg|\int_t^{\bar{t}}\frac{1}{\bar{t}-s}E_{\mu^{t,\bar{t}}_{s}}[\mathcal{L}_{G_\xi}\phi^{t,\bar{t}}_{s}(\xi)]ds \bigg|\\
&\le&\int_t^{\bar{t}}\frac{4}{\bar{t}-s}[\partial_x\phi^{t,\bar{t}}_s]_\alpha ds  \times \mathbb{E}[|\xi|^{1+\alpha}]\\
&=&\int_t^{\bar{t}}4(\bar{t}-s)^{\alpha}[\partial_xv(\cdot, s)]_\alpha ds  \times \mathbb{E}[|\xi|^{1+\alpha}]\\
&\le& 4\delta^{\alpha}\int_t^{\bar{t}}[\partial_xv(\cdot, s)]_\alpha ds \times \mathbb{E}[|\xi|^{1+\alpha}].
\end {eqnarray*}
\end {proof}

\begin {theorem} \label {Stein-Est} Let $(X_i)_{i\ge 1}$ be a sequence of independent and identically distributed random variables  under a sublinear expectation $\hat{\mathbf{E}}$ such that $\hat{\mathbf{E}}[|X_1|^{1+\alpha}]<\infty$ for some $\alpha\in (0,1]$. Let $v\in C_b^{1,\alpha}(\mathbb{R}\times\mathbb{R}_+)$  be a solution to Equ. (\ref {gpE})  with $p(a)=\mathbb{E}[aX_1]$. Set $S_n=\frac{X_1+\cdots+X_n}{n}$. Then we have
\begin {eqnarray}
\begin {split} \label {Est1-m}
v(0,1)-\mathbb{E}[\phi(S_n)]\ge & - 4n^{-\alpha} \int_{0}^{1}[\partial_xv(\cdot, s)]_\alpha ds \times \mathbb{E}[|X_1|^{1+\alpha}]\\
& -\Sigma_{i=1}^n\int_{\frac{i-1}{n}}^{\frac{i}{n}}\mathbb{E}[-f((\frac{i}{n}- s)X_1, s)] ds.
\end {split}
\end {eqnarray}
and
\begin {eqnarray}
\begin {split}
v(0,1)-\mathbb{E}[\phi(S_n)]\le  & 4n^{-\alpha} \int_{0}^{1}[\partial_xv(\cdot, s)]_\alpha ds \times \mathbb{E}[|X_1|^{1+\alpha}]\\
&+\Sigma_{i=1}^n\int_{\frac{i-1}{n}}^{\frac{i}{n}}\mathbb{E}[f((\frac{i}{n}- s)X_1, s)] ds.
\end {split}
\end {eqnarray}
\end {theorem}

\begin {proof}  We only prove (\ref {Est1-m}).
Set, for $1\le i\le n$,
\begin{eqnarray*} S_{0,n}=0, \ S_{i,n}=\sum_{k=1}^i\frac{X_k}{n},
\end{eqnarray*}
and, for $0\le i\le n$,
 \[A_{i,n}=\hat{\mathbf{E}}[v(S_{i,n},1-\frac{i}{n})].\]

Note that $A_{n,n}=\hat{\mathbf{E}}[\phi(S_n)]$, $A_{0,n}=v(0,1)$, and
\begin{eqnarray*}
\hat{\mathbf{E}}[\phi(S_n)]-v(0,1)&=&\sum_{i=1}^n (A_{i,n}-A_{i-1,n})\\
&=&\sum_{i=1}^n\big(\hat{\mathbf{E}}[b_{i,n}(S_{i-1,n})]-\hat{\mathbf{E}}[c_{i,n}(S_{i-1,n})]\big)\\
&\le& \sum_{i=1}^n\sup_{x\in\mathbb{R}}\big(b_{i,n}(x)-c_{i,n}(x)\big),
\end{eqnarray*} where $b_{i,n}(x)=\hat{\mathbf{E}}[v(x+\frac{X_i}{n}, 1-\frac{i}{n})]$ and $c_{i,n}(x)=v(x, 1-\frac{i-1}{n})$.  Then, by (\ref {Est1-local}) in Proposition \ref {Stein-pro}, we have

\begin{eqnarray*}b_{i,n}(x)-c_{i,n}(x)\le & &4n^{-\alpha} \int_{1-\frac{i}{n}}^{1-\frac{i-1}{n}}[\partial_xv(\cdot, s)]_\alpha ds \times \mathbb{E}[|X_1|^{1+\alpha}]\\
& &+\int_{1-\frac{i}{n}}^{1-\frac{i-1}{n}}\mathbb{E}[-f((\frac{n-i+1}{n}- s)X_1, s)] ds.
\end {eqnarray*}

Noting that the right hand of the above inequality is independent of  $x$, we get
\begin{eqnarray*}A_{i,n}-A_{i-1,n}\le & & 4n^{-\alpha} \int_{1-\frac{i}{n}}^{1-\frac{i-1}{n}}[\partial_xv(\cdot, s)]_\alpha ds \times \mathbb{E}[|X_1|^{1+\alpha}]\\
& &+\int_{1-\frac{i}{n}}^{1-\frac{i-1}{n}}\mathbb{E}[-f((\frac{n-i+1}{n}- s)X_1, s)] ds
\end {eqnarray*} and 
\begin {eqnarray}
\begin {split}
\mathbb{E}[\phi(S_n)]-v(0,1)\le &  4n^{-\alpha} \int_{0}^{1}[\partial_xv(\cdot, s)]_\alpha ds \times \mathbb{E}[|X_1|^{1+\alpha}]\\
& +\Sigma_{i=1}^n\int_{\frac{i-1}{n}}^{\frac{i}{n}}\mathbb{E}[-f((\frac{i}{n}- s)X_1, s)] ds.
\end {split}
\end {eqnarray}
\end {proof}

Define $\underline {\mu}=-\hat{\mathbf{E}}[-X_1]$, $\overline {\mu}=\hat{\mathbf{E}}[X_1]$. Then $v(x,t):=\sup_{y\in[\underline{\mu}, \overline{\mu}]}\phi(x+ty)$ is the viscosity solution to the following equation:
\begin {eqnarray} \label {pE}
\begin {split}
\partial_tv(x,t)-p(\partial_xv(x,t)) & =0, \ (x,t)\in \mathbb{R}\times (0,\infty),\\
v(x,0) & =\phi(x).
\end {split}
\end {eqnarray}
So, if $v\in C_b^{1,\alpha}(\mathbb{R})$, it follows from Theorem \ref {Stein-Est} that
\[\bigg|\mathbb{E}[\phi(S_n)]-\sup_{y\in[\underline{\mu}, \overline{\mu}]}\phi(y)\bigg|\le   4n^{-\alpha} \int_{0}^{1}[\partial_xv(\cdot, s)]_\alpha ds \times \mathbb{E}[|X_1|^{1+\alpha}],\] which gives an estimate of the convergence rate for Peng's LLN under sublinear expectations. Unfortunately, generally the solution $v$ to Equ. (\ref {pE}) is just Lipschitz continuous. In order to give error estimates for Peng's LLN, we need to consider proper smooth approximations of the solutions to Equ. (\ref {pE}), which turn out to be solutions to equations of type (\ref {gpE}).

\section {\label {section-gpE}Smooth approximations of the solutions to Equ. (\ref {pE})}
In this section,  we construct two sequences of smooth approximations of the solutions to Equ. (\ref {pE}), which are motivated by Krylov (2018) (\cite {Kr}).

Take a nonnegative $\zeta \in C _ { 0 } ^ { \infty } \left( \mathbb { R } ^ { 2 } \right)$ with unit integral and support in $\{(x,t): |x|<1, 0<t<1 \}$ and for $\varepsilon \in ( 0,1 )$ introduce $\zeta _ { \varepsilon } ( x , t) = \varepsilon ^ { - 2 } \zeta \left( x / \varepsilon  , t / \varepsilon \right)$. Next, for locally summable $u(x,t)$ use the notation $$u_\varepsilon  =  u * \zeta _ { \varepsilon }.$$ In the sequel, we shall denote by $C_\zeta$ a constant depending only on $\zeta$.

\begin {lemma} \label {lem-gpE-App1} For $\phi$ satisfying $|\phi(x)-\phi(y)|\le |x-y|$, set $v(x,t)=\sup_{y\in{[\underline{\mu}, \overline{\mu}]}}\phi(x+ty)$. Then $v_\varepsilon$ satisfies
\begin {eqnarray} \label {gpE-App1}
\partial_tv_\varepsilon(x,t)-p(\partial_xv_\varepsilon(x,t)) & = f_\varepsilon(x,t)\ge0, \ (x,t)\in \mathbb{R}\times (\varepsilon,\infty).
\end {eqnarray}
\end {lemma}
\begin {proof}
Note that $v(x, t+\delta)=\sup_{y\in{[\underline{\mu}, \overline{\mu}]}}v(x+\delta y,t)$. It follows that, for $t>\varepsilon$,
$$v_\varepsilon(x, t+\delta)\ge\sup_{y\in{[\underline{\mu}, \overline{\mu}]}}v_\varepsilon(x+\delta y,t),$$ which implies that, for any $t>\varepsilon$,
$$\partial_tv_\varepsilon-p(\partial_xv_\varepsilon)=:f_\varepsilon\ge0.$$ \end {proof}

Let $(X_k)_{k\ge 1}$ be a sequence of independent and identically distributed random variables  under a sublinear expectation $\hat{\mathbf{E}}$ with $\hat{\mathbf{E}}[|X_1|^{2}]<\infty$. For $n,k\in\mathbb{N}$,  define recursively, $v_n(x,0)=\phi(x)$,
$$v_n(x,\frac{k}{n})=\hat{\mathbf{E}}[v_n(x+\frac{X_k}{n}, \frac{k-1}{n})].$$
Define $T _ { n } = \{ k / n : k \in \mathbb{N} \}$. Extend $v_n(x,t)$ to the whole of $\mathbb{R}_+$ keeping its values on $T_n$ and making it
constant on each interval $( k / n , ( k + 1 ) / n ]$ and equal there to $v _ { n } (x, (k+1) / n)$. We
keep the notation $v_n$ for the extended function.

\begin {lemma} \label {lem-gpE-App2} For $\phi$ satisfying $|\phi(x)-\phi(y)|\le |x-y|$, let $v_n$ be the function defined above. Then $v_{n, \varepsilon}$ satisfies
 \begin {eqnarray}\label {gpE-App2}\partial_t v_{n,\varepsilon}(x,t)-p(\partial_xv_{n,\varepsilon}(x,t))=:f_{n,\varepsilon}(x,t), \ \ (x,t)\in \mathbb{R}\times (\varepsilon+\frac{1}{n},\infty),
 \end {eqnarray}
where $f_{n,\varepsilon}(x,t)\ge-\frac{C_\zeta}{n\varepsilon}(\frac{1}{2}\hat{\mathbf{E}}[|X_1|^2]+3\hat{\mathbf{E}}[|X_1|])$ and $C_{\zeta}$  depends only on $\zeta$.
\end {lemma}

\begin {proof} First,  it is easily seen that, for $x,y \in \mathbb{R}$ and $s,t \in T_n$,
\[|v_n(x,t)-v_n(y,t)|\le |x-y|, \ |v_n(x,t)-v_n(x,s)|\le \hat{\mathbf{E}}[|X_1|]|t-s|.\]
By the definition of the extension of $v_n$,  for $x,y\in \mathbb{R}$
and $s,t\in \mathbb{R}_+$ we have
\begin {eqnarray}\label {Est-vnx}\left| v _ { n } (x,t) - v _ { n } (x,s) \right| &\leq& \hat{\mathbf{E}}[|X_1|] (| t - s |  + n ^ { -1 }),\\
\label {Est-vnt}\left| v _ { n } (x,t) - v _ { n } (y,t) \right| &\leq& | x - y |.
\end {eqnarray}
It follows from the definition of $v_n$, for $t\ge \frac{1}{n}$,
$$v _ { n } (x,t) = \hat{\mathbf{E}}[ v _ { n } (x + X_1 / n , t - 1 / n )].$$ Then, for $t\ge \frac{1}{n}+\varepsilon$,
$$v _ { n, \varepsilon } (x,t) \geq \hat{\mathbf{E}}[ v _ { n,  \varepsilon }  (x + X_1 /n, t - 1 / n)].$$
Taylor's  formula gives \[v _ { n,  \varepsilon }  (x + \frac{X_1}{n}, t -\frac{1}{n})=v _ { n,  \varepsilon }  (x, t - \frac{1}{n})+\partial_xv _ { n,  \varepsilon }  (x, t)\frac{X_1}{n}+I_{ n,  \varepsilon}+J_{n,\varepsilon}\]
 with $ I_{n,\varepsilon} := \frac{1}{2}\partial_x^2v_{n,\varepsilon}(x+\frac{\theta X_1}{n}, t-\frac{1}{n})\frac{|X_1|^2}{n^2}, \ J_{n,\varepsilon}:=(\partial_xv _ { n,  \varepsilon }  (x, t -\frac{1}{n})-\partial_xv _ { n,  \varepsilon }  (x, t ))\frac{X_1}{n}$,
 \[v _ { n,\varepsilon } (x,t)-v _ { n,  \varepsilon }  (x, t - \frac{1}{n})=\partial_tv_{n,\varepsilon}(x,t)\frac{1}{n}+K_{n,\varepsilon}\]
 with $K_{n,\varepsilon} := \big(\partial_t v_{n,\varepsilon}(x,t-\frac{\theta}{n})-\partial_t v_{n,\varepsilon}(x,t)\big)\frac{1}{n}.$
  
  By (\ref{Est-vnx}), (\ref{Est-vnt}) and Lemma 2.3 in \cite {Kr},
  \begin {eqnarray*}
 \hat{\mathbf{E}}[|I_{ n,  \varepsilon}|] &\le& \frac{1}{2}C_\zeta\frac{1}{n^2\varepsilon}\hat{\mathbf{E}}[|X_1|^2]\\
 \hat{\mathbf{E}}[ |J_{n,\varepsilon}|]&\le& C_\zeta\frac{1}{n^2\varepsilon} \hat{\mathbf{E}}[ |X_1|]\\
 |K_{n,\varepsilon}|&\le& 2C_\zeta\frac{1}{n^2\varepsilon}\hat{\mathbf{E}}[|X_1|],
 \end {eqnarray*} 
  where  $C_{\zeta}$  is a constant depending only on $\zeta$.

Hence, for $t\ge \frac{1}{n}+\varepsilon$, \[\partial_t v_{n,\varepsilon}(x,t)-p(\partial_xv_{n,\varepsilon}(x,t))=:f_{n,\varepsilon}(x,t),\] 
where $f_{n,\varepsilon}(x,t)\ge-\frac{C_\zeta}{n\varepsilon}(\frac{1}{2}\hat{\mathbf{E}}[|X_1|^2]+3\hat{\mathbf{E}}[|X_1|])$.\end {proof}

\section {Convergence rate of LLN under sublinear expectations } 

\begin {theorem} Let $(X_k)_{k\ge 1}$ be a sequence of independent and identically distributed random variables  under a sublinear expectation $\hat{\mathbf{E}}$ with $\hat{\mathbf{E}}[|X_1|^{2}]<\infty$. Set $S_n=\frac{X_1+\cdots+X_n}{n}$, $\underline{\mu}=-\hat{\mathbf{E}}[-X_1]$, $\overline{\mu}=\hat{\mathbf{E}}[X_1]$. Then
\begin {eqnarray}
\sup_{|\phi|_{Lip}\le1}\bigg| \hat{\mathbf{E}}[\phi(S_n)]-\sup_{y\in[\underline{\mu}, \overline{\mu}]}\phi(y)\bigg|\le C n^{-1/2},
\end {eqnarray} where $C$ is a constant depending only on $\hat{\mathbf{E}}[|X_1|^2]$.

\end {theorem}

\begin {proof}
For $\phi$ satisfying $|\phi(x)-\phi(y)|\le |x-y|$, let $v$, $v_{\varepsilon}$, $v_n$, $v_{n, \varepsilon}$ be the functions defined in Section \ref {section-gpE}.

STEP 1. Estimating  $\hat{\mathbf{E}}[\phi(S_n)]-v(0, 1)$ from above.

It is easily seen that $|v(x,t)-v(y,t)|\le |x-y|$ and $|v(x,t)-v(x,s)|\le m_p|t-s|$ with $m_p=\max\{|\underline{\mu}|, |\overline{\mu}|\}$. 

By Theorem \ref {Stein-Est} and Lemma \ref {lem-gpE-App1} , we have
\begin {eqnarray*}v_\varepsilon(0, 1+\varepsilon)-\hat{\mathbf{E}}[v_\varepsilon(S_n, \varepsilon)]&\ge&-\frac{4}{n}\hat{\mathbf{E}}[|X_1|^2]\int_0^1\|\partial_x^2v_\varepsilon(\cdot, s+\varepsilon)\|_\infty ds\\
&\ge& -\frac{4C_\zeta}{n\varepsilon}\hat{\mathbf{E}}[|X_1|^2].
\end {eqnarray*}
Note that $|v_\varepsilon(0, 1+\varepsilon)-v(0,1)|\le |v_\varepsilon(0, 1+\varepsilon)-v_\varepsilon(0,1)|+|v_\varepsilon(0, 1)-v(0,1)|\le (2m_p+1)\varepsilon$, and
$|\hat{\mathbf{E}}[v_\varepsilon(S_n, \varepsilon)]-\phi(S_n)|\le |\hat{\mathbf{E}}[v_\varepsilon(S_n, \varepsilon)]-v(S_n, \varepsilon)]|+|\hat{\mathbf{E}}[v(S_n, \varepsilon)]-\phi(S_n)]|\le (2m_p+1)\varepsilon$. So
\begin {eqnarray*}v(0, 1)-\hat{\mathbf{E}}[\phi(S_n)]\ge -\frac{4C_\zeta}{n\varepsilon}\hat{\mathbf{E}}[|X_1|^2]-(4m_p+2)\varepsilon.
\end {eqnarray*}
Taking $\varepsilon=n^{-1/2}$, we have
\begin {eqnarray*}v(0, 1)-\hat{\mathbf{E}}[\phi(S_n)]\ge -4(C_\zeta\hat{\mathbf{E}}[|X_1|^2]+m_p+1/2) n^{-1/2}.
\end {eqnarray*}

STEP 2. Estimating  $\hat{\mathbf{E}}[\phi(S_n)]-v(0, 1)$ from below.

Let $\xi_1, \xi_2, \cdots$ be a sequence of independent and $p$-maximally distributed random variables under the sublinear expectation $\hat{\mathbf{E}}$, that is, $\hat{\mathbf{E}}[\varphi(\xi_1)]=\max_{y\in[\underline{\mu}, \overline{\mu}]}\varphi(y)$, for $\varphi\in C_{b, Lip}(\mathbb{R}).$

Set $S^\xi_n=\frac{\xi_1+\cdots+\xi_n}{n}$.

By Theorem \ref {Stein-Est} and Lemma \ref {lem-gpE-App2} , we have
\begin {eqnarray*}& &v_{n,\varepsilon}(0, 1+\varepsilon+\frac{1}{n})-\hat{\mathbf{E}}[v_{n,\varepsilon}(S^\xi_n, \varepsilon+\frac{1}{n})]\\
&\ge&-\frac{4}{n}\hat{\mathbf{E}}[|\xi_1|^2]\int_0^1\|\partial_x^2v_{n,\varepsilon}(\cdot, s+\varepsilon+\frac{1}{n})\|_\infty ds-\frac{C_\zeta}{n\varepsilon}(\frac{1}{2}\hat{\mathbf{E}}[|X_1|^2]+3\hat{\mathbf{E}}[|X_1|])\\
&\ge& -\frac{C_\zeta}{n\varepsilon}(4m_p^2+\frac{1}{2}\hat{\mathbf{E}}[|X_1|^2]+3\hat{\mathbf{E}}[|X_1|]).
\end {eqnarray*}

The last inequality follows from the fact that $\left| v _ { n } (x,t) - v _ { n } (y,t) \right| \leq | x - y |$, and consequently that
\[\|\partial_x^2v_{n,\varepsilon}(\cdot, s+\varepsilon+\frac{1}{n})\|_\infty\le \frac{C_\zeta}{\varepsilon}.\]
Note that
\begin {eqnarray*}& &|v_{n,\varepsilon}(0, 1+\varepsilon+\frac{1}{n})-\hat{\mathbf{E}}[\phi(S_n)]|\\
&=&|v_{n,\varepsilon}(0, 1+\varepsilon+\frac{1}{n})-v_n(0,1)|\\
&\le&|v_{n,\varepsilon}(0, 1+\varepsilon+\frac{1}{n})-v_n(0,1+\varepsilon+\frac{1}{n})|+|v_{n}(0, 1+\varepsilon+\frac{1}{n})-v_n(0,1)|\\
&\le&(1+2\hat{\mathbf{E}}[|X_1|])\varepsilon+3\hat{\mathbf{E}}[|X_1|]\frac{1}{n}.
\end {eqnarray*}
The last inequality follows from (\ref {Est-vnx}) and (\ref {Est-vnt}).

On the other hand,
\begin {eqnarray*}& &|\hat{\mathbf{E}}[v_{n,\varepsilon}(S^\xi_n, \varepsilon+\frac{1}{n})]-v(0,1)|\\
&=&|\hat{\mathbf{E}}[v_{n,\varepsilon}(S^\xi_n, \varepsilon+\frac{1}{n})]-\hat{\mathbf{E}}[\phi(S^\xi_n)]|\\
&\le&|\hat{\mathbf{E}}[v_{n,\varepsilon}(S^\xi_n, \varepsilon+\frac{1}{n})]-\hat{\mathbf{E}}[v_{n}(S^\xi_n, \varepsilon+\frac{1}{n})]|+|\hat{\mathbf{E}}[v_{n}(S^\xi_n, \varepsilon+\frac{1}{n})]-\hat{\mathbf{E}}[v_{n}(S^\xi_n, 0)]|\\
&\le&(1+2\hat{\mathbf{E}}[|X_1|])\varepsilon+3\hat{\mathbf{E}}[|X_1|]\frac{1}{n}.
\end {eqnarray*}
The last inequality also follows from (\ref {Est-vnx}) and (\ref {Est-vnt}).

So $\hat{\mathbf{E}}[\phi(S_n)]-v(0,1)\ge -\frac{C_\zeta}{n\varepsilon}(4m_p^2+\frac{1}{2}\hat{\mathbf{E}}[|X_1|^2]+3\hat{\mathbf{E}}[|X_1|])-2(1+2\hat{\mathbf{E}}[|X_1|])\varepsilon-6\hat{\mathbf{E}}[|X_1|]\frac{1}{n}$.
Taking $\varepsilon=n^{-1/2}$, we have
\begin {eqnarray*}\hat{\mathbf{E}}[\phi(S_n)-v(0,1)&\ge&- n^{-1/2}\bigg(C_\zeta(4m_p^2+\frac{1}{2}\hat{\mathbf{E}}[|X_1|^2]+3\hat{\mathbf{E}}[|X_1|])+2+10\hat{\mathbf{E}}[|X_1|]\bigg).
\end {eqnarray*} 
\end {proof}

\begin {remark} \cite {FPSS17} gave the following rate of convergence for Peng's LLN under sublinear expectations:
$$
\hat{\mathbf{E}}\left[d_{[\underline{\mu}, \overline{\mu}]}^{2}\left(S_{n}\right)\right] \leqslant \frac{2\left[\overline{\sigma}^{2}+(\overline{\mu}-\underline{\mu})^{2}\right]}{n},
$$ where $d_{[\underline{\mu}, \overline{\mu}]}(x):=\inf_{y\in[\underline{\mu}, \overline{\mu}]}|x-y|$. Roughly speaking, this estimate corresponds to the upper estimate of $\hat{\mathbf{E}}[\phi(S_n)]-\sup_{y\in[\underline{\mu}, \overline{y}]}\phi(y)$. Set $\phi(x):=d_{[\underline{\mu}, \overline{\mu}]}(x)$. Noting that $\sup_{y\in[\underline{\mu}, \overline{\mu}]}\phi(y)=0$, we have $\hat{\mathbf{E}}[d_{[\underline{\mu}, \overline{\mu}]}(S_n)]=\hat{\mathbf{E}}[\phi(S_n)]-\sup_{y\in[\underline{\mu}, \overline{\mu}]}\phi(y)$.
\end {remark}

\section {Stein type characterization for maximal distribution $\mathcal{M}_p$}

In this section, we give a Stein type characterization for the maximal distribution $\mathcal{M}_p$. Under a sublinear expectation, one can not derive from such  characterization the Stein equation like the linear case, but this charcterization is still instructive on how to establish Stein's method under sublinear expectations. 
\begin {proposition} Let $\mathcal{N}[\phi]=\sup_{\mu\in \Theta}\mu[\phi]$ be a sublinear expectation on $C_{b, Lip}(\mathbb{R})$. Then $\mathcal{N}[\phi]=\sup_{y\in[\underline{\mu}, \overline{\mu}]}\phi(y)$ if and only if for any $\varphi\in C_b^1(\mathbb{R})$ and $\mu\in \Theta_\varphi$
\[E_\mu[p(\varphi'(\xi))-\xi\varphi'(\xi)]=0,\]
where $p(a)=\sup_{y\in[\underline{\mu},\overline{\mu}]}(y a)$, $a\in\mathbb{R}$.
\end {proposition}

\begin {proof}

(Necessity )

For $\varphi\in C_b^1(\mathbb{R})$, assume $\varphi(x_\varphi)=\sup_{y\in[\underline{\mu}, \overline{\mu}]}\phi(y)$ for some $x_\varphi$. Clearly we have $p(\varphi'(x_\varphi))-x_\varphi\varphi'(x_\varphi)=0$.

(Sufficiency)

For $\phi\in C_{b,Lip}(\mathbb{R})$, let $v_\varepsilon$ be the function defined in Section \ref {section-gpE}. By Lemma \ref {SteinEquation-LLN} and Lemma \ref {lem-gpE-App1} , it follows from the assumption that
\[v_\varepsilon(0,1+\varepsilon)-\mathcal{N}[v_\varepsilon(\cdot, \varepsilon)]\ge0.\]
Let $\varepsilon$ go to zero, we get
\begin {eqnarray*}\mathcal{N}[\phi]\le\sup_{y\in[\underline{\mu}, \ \overline{\mu}]}\phi(y).
\end {eqnarray*}
This implies that 
\begin {eqnarray} \label {SteinChara} \mu([\underline{\mu}, \overline{\mu}])=1, \  \textit{for any} \ \mu\in\Theta.
\end {eqnarray} Now let us prove the reversed inequality. For any $y\in(\underline{\mu}, \overline{\mu})$, set $\varphi_y(x):=\exp\{-\frac{(x-y)^2}{2}\}$.
Note that $\varphi_y$ belongs to $C_b^1(\mathbb{R})$ with $\varphi_y'(x)>0$ on $(-\infty, y)$ and $\varphi_y'(x)<0$ on $(y, +\infty)$. Then, for  $\mu\in\Theta_{\varphi_y}$, it follows from the assumption and (\ref{SteinChara}) that
\begin{eqnarray}
p(\varphi_y'(x))=x\varphi_y'(x), \ \mu-a.s..
\end{eqnarray}
Note that, on $(-\infty, y)$,
\[p(\varphi_y'(x))-x\varphi_y'(x)=(\overline {\mu}-x)\varphi_y'(x)>0,\]
and, on $(y, +\infty)$,
\[p(\varphi_y'(x))-x\varphi_y'(x)=(\underline {\mu}-x)\varphi_y'(x)>0.\]
Hence, we have $\mu=\delta_y\in\Theta.$
\end {proof}

\section*{Acknowledgements}

The author is financially supported by NCMIS; NSFCs (No. 11871458 \& No. 11688101); and
Key Research Program of Frontier Sciences, CAS (No. QYZDB-SSW-SYS017).


\renewcommand{\refname}{\large References}{\normalsize \ }

\end{document}